\documentclass[11pt]{elsarticle}
\UseRawInputEncoding
\usepackage{amssymb}
\usepackage{multirow}
\usepackage{subcaption}
\usepackage{float}
\usepackage{booktabs}
\usepackage[utf8]{inputenc}
\usepackage[T1]{fontenc}
\usepackage{amsfonts,amssymb,latexsym}
\usepackage[centertags]{amsmath}
\usepackage{algorithmic}
\usepackage{newlfont}
\usepackage{graphics}
\usepackage{graphicx}
\usepackage{color}
\usepackage{tikz}
\usepackage{array}
\usepackage{booktabs}
\usepackage{newlfont}
\usepackage{graphics}
\usepackage{color}
\usepackage{booktabs}
\usepackage{tikz}
\usepackage{amsthm}
\newtheorem{theorem}{Theorem}[section]
\newtheorem{corollary}{Corollary}[theorem]
\newtheorem{lemma}{Lemma}[section]
\newtheorem{definition}{Definition}[section]

\newtheorem{remark}{Remark}[section]
\makeatletter
\def\ps@pprintTitle{
  \let\@oddhead\@empty
  \let\@evenhead\@empty
  \let\@oddfoot\@empty
  \let\@evenfoot\@oddfoot
}
\input{epsf}
\begin{document}
\begin{frontmatter}
\title{A shape preserving quasi-interpolation \\ operator based on a new transcendental RBF}
\author[1]{Mohammad Heidari}
\ead{Std_M.Heidari@khu.ac.ir}
\author[1]{Maryam Mohammadi\corref{c2}}
\ead{m.mohammadi@khu.ac.ir}
\cortext[c2]{Corresponding author}
\author[2]{Stefano De Marchi}
\ead{demarchi@math.unipd.it}
\address[1]{Faculty of Mathematical Sciences and Computer, Kharazmi University, Teheran, Iran}
\address[2]{Department of Mathematics "Tullio Levi-Civita", University of Padova, Italy}
\begin{abstract}
It is well-known that the univariate Multiquadric quasi-interpolation operator is constructed based on the piecewise linear interpolation  by $|x|$.
In this paper, we first introduce a new transcendental RBF based on  the  hyperbolic  tangent  function as an smooth approximant to $\phi(r)=r$ with higher accuracy and better convergence properties than the MQ RBF $\sqrt{r^2+c^2}$. Then the Wu–Schaback's quasi-interpolation formula is rewritten using the proposed RBF. It
preserves convexity and monotonicity. We prove that the proposed scheme converges with a rate of $O(h^2)$. So it has a higher degree of smoothness. Some numerical experiments are given in order to demonstrate the efficiency and accuracy of the method.
\end{abstract}
\end{frontmatter}
\section{Introduction}
Given a set of $n$ distinct (scattered) points $\{x_j\}_{j=0}^n\in\Omega\subseteq\mathbb{R}^d$ and corresponding data values $\{f_j\}_{j=0}^n\in\mathbb{R}$, a standard way to interpolate a function $f\in C^1:\Omega\rightarrow\mathbb{R}$ is by using
\begin{eqnarray}\label{meq1}
\mathcal{L}f(x)=\sum_{j=0}^n\lambda_j\mathcal{X}(x-x_j),
\end{eqnarray}
with the coefficients $\lambda_j$ determined by the interpolation conditions $\mathcal{L}f(x_j)=f_j,~j=0,\ldots,n,$ where $\mathcal{X}(\cdot)$ is an interpolation kernel. Many authors use Radial Basis Functions (RBFs)
to solve the interpolating problem (\ref{meq1}), that is $\mathcal{X}(x-x_j)=\phi(\|x-x_j\|),$ ($\|\cdot \|$ is the Euclidean norm)
with $\phi: [0,\infty) \rightarrow \mathbb{R}$, is some radial function  \cite{wendland}. Then, the coefficients $\lambda_j$ are determined solving a symmetric linear system $A\lambda=f$, where $A=\left[\phi(\|x_i-x_j\|)\right]_{0\leq i,j\leq n}.$ RBF method provides excellent interpolants for high dimensional scattered data sets. The corresponding theory had been extensively studied by many researchers (see e.g \cite{2high, 4high, 3high, 5high, 6high, 7high, 8high, wendland, 10high, 11high}). That is why in the last few decades, RBFs have been widely applied in a number of fields such as multivariate function approximation, neural networks and solution of differential
and integral equations (see e.g \cite{14high, 16high, 17high, 18high, 19high, 12high, 13high, mohammadi-meshless, mohammadi-onthe,  21high,  mohammadi-amc}).
The Multiquadric (MQ) RBF 
\begin{eqnarray}\label{meq0}
\phi_j(x)=\sqrt{{\|x-x_j\|}^2+c^2},
\end{eqnarray}
proposed by Hardy \cite{Hardy1}, is undoubtedly the most popular RBF that is used
in many applications and is representative of the class of global infinitely differentiable RBFs. Hardy \cite{2ref} summarized the achievement of study of MQ from 1968 to 1988 and showed that it can be applied in hydrology, geodesy, photogrammetry, surveying and mapping, geophysics, crustal movement, geology, mining and so on.
In the survey paper \cite{Franke1}, Franke pointed out that MQ interpolation was the best among 29 scattered data interpolation
methods in terms of timing, storage, accuracy, visual pleasantness of surface reconstruction and ease to implement.
The existence of the solution of the associated interpolation problem was shown later on by Micchelli \cite{5high}. 
Although the MQ interpolation is always solvable, the resulting matrix quickly becomes ill-conditioned as the number of points increases.
Researchers concentrated on a weaker form of (\ref{meq1}), known as {\it quasi-interpolation}, that holds only for polynomials of some low degree $m$, i.e.,
$$\mathcal{L}p_m(x_j)=p_m(x_j),\quad\forall p_m\in\Pi_m^d,$$
for all $0\leq j\leq n$, where $\Pi_m^d$ denotes the space of polynomials of degree less and equal to $m$ in $\mathbb{R}^d$. Beatson and Powell \cite{Beatson1, Powell1} first proposed a univariate quasi-interpolation formula
where $\mathcal{X}$ in (\ref{meq1}), is a linear combination of MQ RBF and low degree polynomials. Their idea is based on the fact that the MQ degenerates to 
$|x-x_j|$, for $c=0$ and $d=1$, hence quasi-interpolation (\ref{meq1}) is the usual piecewise linear interpolation which reproduces linear polynomials as $c$ tends to zero.
However, their operator requires the approximation of the derivatives of the function at endpoints, which is not convenient for practical purposes.
Thus, Wu and Schaback \cite{schaback-quasi} constructed another univariate MQ quasi-interpolation operator with without the use of derivatives at the endpoints. Given data
\begin{eqnarray*}
a=x_{0}<x_{2}<\dotsb< x_{n} = b \qquad h := \max_{2\leq j\leq n} (x_{j} - x_{j-1}),
\end{eqnarray*}
Wu–Schaback's MQ quasi-interpolation formula is 
\begin{eqnarray}\label{meq2}
(\mathcal{L} _{{MQ}}f)(x)=f_{0}\alpha_{0}(x)+f_{1}\alpha_{1}(x)+\sum_{j=2}^{n-2}f_{j}\psi_{j}(x)+f_{n-1}\alpha_{n-1}(x)+f_{n}\alpha_{n}(x)\label{eq04}
\end{eqnarray}
where
\begin{align}
\alpha_{0}(x)&=\dfrac{1}{2}+\dfrac{\phi_{1}(x)-(x-x_{0})}{2(x_{1}-x_{0})}, \notag \\
\alpha_{1}(x)&=\dfrac{\phi_{2}(x)-\phi_{1}(x)}{2(x_{2}-x_{1})}-\dfrac{\phi_{1}(x)-(x-x_{0})}{2(x_{1}-x_{0})}, \notag\\
\alpha_{n-1}(x)&=\dfrac{(x_{n}-x)-\phi_{n-1}(x)}{2(x_{n}-x_{n-1)}}-\dfrac{\phi_{n-1}(x)-\phi_{n-2}(x)}{2(x_{n-1}-x_{n-2})}, \notag\\
\alpha_{n}(x)&=\dfrac{1}{2}+\dfrac{\phi_{n-1}(x)-(x_{n}-x)}{2(x_{n}-x_{n-1})}, \notag \\
\psi_{j}(x)&=\dfrac{\phi_{j+1}(x)-\phi_{j}(x)}{2(x_{j+1}-x_{j})}-\dfrac{\phi_{j}(x)-\phi_{j-1}(x)}{2(x_{j}-x_{j-1)}},\qquad 2\leq j \leq n-2.\notag
\end{align}
The main advantage of this formula is that it does not require the solution of any linear system.
Instead, the formula uses the function values $f_j$ at $x_j$ as its coefficients. The drawback
is that instead of $c=O(h)$, one needs to use a smaller shape parameter
$c^2|\log c|=O(h^2)$ in order to achieve quadratic convergence, resulting in a lower
smoothness. Note that for $c=0$, the basis functions given in quasi-interpolant $\mathcal{L}_{MQ}f$ are cardinal with respect to $\{x_j\}_{j=0}^n$. For a general quasi-interpolation operator $\mathcal{L}$ we can state the following definitions.
\begin{definition}
The quasi-interpolation operator $\mathcal{L}$ constructed at the data points $\{(x_j,f_j)\}$, is called to be monotonicity-preserving, if the first order divided difference
$f[x_j, x_{j+1}]$ is nonnegative (non-positive)
implies that
$\left(\mathcal{L}f\right)'$ is also nonnegative (non-positive). 
\end{definition}

\begin{definition}
The quasi-interpolation operator $\mathcal{L}$ constructed at the data points $(x_j,f_j)$, is called to be
convexity-preserving if the second order divided difference
$f[x_{j-1}, x_j, x_{j+1}]$ is nonnegative (non-positive, zero)
implies that
$\left(\mathcal{L}f\right)''$ is also nonnegative (non-positive, zero). 
\end{definition}

Since $\sqrt{x^2+c^2}$ tends to $|x|$ as $c$ tends to zero, and radial basis interpolation (as well
as the quasi-interpolation) based on $|x|$ is piecewise linear, Wu and Schaback claimed that the shape-preserving
properties of piecewise linear interpolation can be expected to hold for quasi-interpolation
with multiquadrics, too. Actually, they first showed that the quasi-interpolation operator of
Beatson and Powell is indeed convexity preserving. Then they proved that the quasi-interpolation operator (\ref{eq04}) is monotonicity and convexity 
preserving. In 2004, Ling \cite{9mam} proposed a multilevel quasi-interpolation operator and proved that it converges with a rate of $O(h^{2.5})\log h$ as $c=O(h)$.
In 2009, Feng and Li \cite{11mam} constructed a shape-preserving quasi-interpolation operator by shifts of cubic MQ functions proving that it can produce
an error of $O(h^2)$ as $c=O(h)$. Wang et al. \cite{12mam} proposed an improved univariate MQ quasi-interpolation operator,
by using Hermite interpolating polynomials, with convergence rate heavily depending on the shape parameter $c$. Jiang et
al. \cite{13mam} proposed two new multilevel univariate MQ quasi-interpolation operators with higher approximation order.

Ling proposed a multidimensional quasi-interpolation operator using the dimension-splitting
multiquadric basis function approach \cite{ling}, and Wu et al. modified their idea by using multivariate divided difference and the idea of the superposition \cite{wu2015}.

Gao and Wu \cite{14mam} studied the quasi-interpolation for the linear functional data rather than the discrete function values.
Moreover, MQ quasi-interpolation has been successfully applied in a wide range of fields. For example, in 2007, Wang and Wu \cite{15mam}
applied the operator (\ref{meq2}) to tackle approximate implicitization of parametric curves. In 2011, Wu \cite{16mam} presented a new approach to construct
the so-called shape preserving interpolation curves based on MQ quasi-interpolation (\ref{meq2}). 
Hon and Wu \cite{17mam}, Chen and
Wu \cite{18mam, 19mam}, Jiang and Wang \cite{20mam}, and other researches provided some successful examples using MQ quasi-interpolation operators to
solve different types of partial differential equations. 

In this paper, in the next section we introduce a new quasi-interpolation operator based on the hyperbolic tangent function, that is the function \begin{equation}\label{gfun}g(x)=x\tanh\left({x \over c}\right),\;\; \,c>0\end{equation} which leads to a smooth and shape preserving interpolation operator with $O(h^2)$ rate of convergence. In section 3, we discuss its accuracy providing an error estimate. Numerical experiments are presented in section 4 with the aim of comparing the accuracy of our quasi-interpolation scheme with that of Wu and Schaback's, and also verifying the
convergence rate of new quasi-interpolation operator by examples. The last section summarizes the conclusion and some further works.
\section{Quasi-interpolation operator based on a new transcendental RBF}
In this section, we first analyse a new approximation of $|x|$ based on the hyperbolic tangent, with better accuracy than the MQ RBF 
$\sqrt{x^2+c^2}$. The general question is, are there any good approximations
of the absolute value function which are smooth? One simple approximation is
MQ RBF $\sqrt{x^2+c^2}$. Carlos Ramirez et al. \cite{3bagul} proved that $\sqrt{x^2+c^2}$ is the most computationally efficient and smooth approximation of $|x|$, 
while S. Voronin et al. \cite{4bagul} proved the following Lemma.
\begin{lemma}
The approximation of $|x|$ by the multiquadrics $g(x) = \sqrt{x^2+c^2}, c \in \mathbb{R}_+$ satisfies
\begin{eqnarray*}
\left| \left| x \right| -\sqrt{x^{2}+c^{2}} \right|\leq c &,\\
\left| x \right| \leq \sqrt{x^{2}+c^{2}} &.  
\end{eqnarray*}
\end{lemma}

As noticed by Gauss in \cite{guass}, the hyperbolic tangent can be written using the continued fraction 
\begin{eqnarray*}
\tanh(x)=\frac{x}{1+\frac{x^2}{3+\frac{x^2}{5+\cdots}}}.
\end{eqnarray*}
This fact shows immediately that the function $\displaystyle g(x)= x \tanh\left(\frac{x}{c}\right)$ is a nonnegative function that indeed can be used to approximate $|x|.$

Since for the hyperbolic tangent   
\begin{eqnarray*}
\displaystyle\lim_{c\rightarrow 0^+} \tanh\left(\frac{x}{c}\right)=\left\{\begin{array}{ll}1, & x>0,\\
0, &x=0\\
-1, & x<0 .
\end{array}.\right.
\end{eqnarray*}
we then have the approximation 
$$x\tanh\left(\frac{x}{c}\right)\approx |x|.$$ Now, we show that the approximation of $|x|$ by $x\tanh\left(\frac{x}{c}\right)$ is more accurate than that given by the multiquadric.

\begin{lemma}
The approximation of $|x|$ by $g(x) = x \tanh\left(\frac{x}{c}\right), c\in \mathbb{R}_+$ satisfies
\begin{align}
\left|  \left| x \right| -x\tanh \left( {\frac {x}{c}} \right) \right|&\leq 0.28c <c,\label{nmeq1}\\
x\tanh \left( {\frac {x}{c}} \right) &\leq \left| x \right|. \label{neq2} 
\end{align}
\end{lemma}
\begin{proof}
The proof of (\ref{nmeq1}) is trivial for  $x=0$. Letting $h(x)=\left| x \right| -x\tanh \left( {\frac {x}{c}}\right)$ that, for $x>0$, becomes $h(x)=x-x\tanh \left( {\frac {x}{c}}\right).$  The maxima and minima of $h$ are those that annihilate
\begin{align*}
h^{\prime}(x)= \left( {\frac {x}{c}} \right) \left( \tanh\left(\frac{x}{c}\right) \right)^{2}-\tanh\left(\frac{x}{c}\right)+\left( 1-\frac {x}{c} \right).
\end{align*}
Setting $\frac{x}{c}=t$, we have
\begin{eqnarray*}
t\tanh^2 (t)-\tanh (t)+\left( 1-t \right)=0\end{eqnarray*}
which reduces to solve $s(t)=t( \tanh(t)+1)-1=0$. The function $s$ on $t\ge 0$ is strictly increasing, with $s(0)=-1$. Hence there exits only one zero. By numerically solving it, we find the value of $t^*=0.6392322714$ then $x^*=0.6392322714c$. When $x<0$, ans so $t<0$, $s(t)<-1$, showing that the value $t^*$ is the only extremal value of $h$. Hence,
\begin{eqnarray*}
h(x^*)=0.2784645427c \simeq 0.28c.
\end{eqnarray*}
To prove \eqref{neq2}, we have
\begin{align*}
x\tanh \left( {\frac {x}{c}} \right) \leq \left| x \right| & \Longleftrightarrow x^{2}\tanh^{2} \left( {\frac {x}{c}} \right) \leq x^{2},\\
&\Longleftrightarrow \tanh^{2} \left( {\frac {x}{c}} \right) \leq 1.
\end{align*}
\end{proof}
\begin{theorem}
The approximation of $|x|$ by $x\tanh(\frac{x}{c})$ is more accurate than that with $\sqrt{x^2+c^2}$.
\end{theorem}
\begin{proof}
It is clear that $$\cosh\left(\frac{x}{c}\right)>\frac{x}{c}.$$ Since $\cosh(x)$ is an even function we have $$\cosh^2\left(\frac{x}{c}\right)>\frac{x^2}{c^2},$$ then
$$x^2\mbox{sech}^2\left(\frac{x}{c}\right)<c^2,$$ which in turn gives $$x^2-x^2\tanh^2\left(\frac{x}{c}\right)<c^2.$$
Then
$$x^2-x^2\tanh^2\left(\frac{x}{c}\right)<c^2=\left(x^2+c^2\right)-x^2.$$
\end{proof}
Moreover, the function $x\tanh \left( {\frac {x}{c}}\right)$ converges to $\left| x \right|$ faster than $\sqrt{x^2+c^2}$ to $\left| x \right|$ by decreasing $c$, as stated in the next Theorem.
\begin{theorem}
If $c\longrightarrow 0^{+}$ then $x\tanh \left( {\frac {x}{c}}\right)-\left| x \right|=o\left(\sqrt{x^{2}+c^{2}}-\left| x \right|\right) $.
\end{theorem}
In fact,
\begin{eqnarray}
\lim_{c\longrightarrow 0^{+}}\dfrac{x\tanh \left( {\frac {x}{c}}\right)-\left| x \right|}{\sqrt{x^{2}+c^{2}}-\left| x \right|}=0
\end{eqnarray}

In order to illustrate the superiority of the new hyperbolic approximation to $|x|$,  $L_\infty$ error norm
$$\max_{1\leq i\leq n}|g(x_i)-|x_i||,$$
and the rate of convergence 
$$r_c=\frac{\log \left(\frac{E_{c_i}}{E_{c_{i-1}}}\right)}{\log \left(\frac{c_i}{c_{i-1}}\right)},$$
for both approximants $x\tanh\left(\frac{x}{c}\right)$ and $\sqrt{x^2+c^2}$ are reported in Table \ref{CRT1}, for $n=100, 200, 400$ equally spaced points 
in $[-10,10]$.
Table \ref{CRT1} shows that $x\tanh\left(\frac{x}{c}\right)$ 
approximates $|x|$ much better than $\sqrt{x^2+c^2}$ while Table \ref{CRT1} and the logarithmic scale plots \ref{loglog} show that the approximant $x\tanh\left(\frac{x}{c}\right)$ has exponential rate of convergence to $|x|$ as $c\rightarrow 0$ instead of $O(c^2)$ provided by $\sqrt{x^2+c^2}$.
\begin{table}[H]
\centering
\textcolor{black}{
\caption{\footnotesize{$L_\infty$ errors and convergence rates for both approximants of $|x|$ for different values of $c$.}}\label{CRT1}
  \scriptsize{
\begin{tabular}{llclclclclcl}
  \hline\\[-2.5mm]
  \hline\\  
 &  &\multicolumn{2}{l}{$\left|  \left| x \right| -\sqrt{x^{2}+c^{2}} \right| $}&\multicolumn{2}{l}{ $  \left|  \left| x \right| -x\tanh \left( {\frac {x}{c}}\right) \right|$}\\
  \cmidrule(lr){3-4}\cmidrule(lr){5-6} 
  $n$ & $c$ & $L_{\infty}$ error &  $r_c$ & $L_{\infty} $ error& $r_c$ \\
 \hline\\
\multirow{4}{*}{100} & 0.1  &4.1127e-02  & --- & 2.3656e-02 & --- \\  
& 0.05  & 1.1698e-02  & 1.813823944 & 3.4922e-03 & 2.759998057 \\ 
& 0.025  & 3.0478e-03 & 1.940421754 & 6.2490e-05 & 5.804367034 \\ 
& 0.0125  & 7.7050e-04 & 1.983901373 & 1.9342e-08 & 11.65767264\\ 
& 0.00625  & 1.9317e-04 & 1.995923901 & 1.8457e-15 & 23.32106557\\ 
 \hline
 \multirow{4}{*}{\mbox{200}} & 0.1 & 6.1665e-02  & --- & 2.6930e-02 & --- \\  
 & 0.05 & 2.0637e-02  & 1.579218611 & 1.1875e-02 & 1.181286716 \\ 
& 0.025 & 5.8753e-03& 1.812498837& 1.7723e-03 & 2.744232777  \\ 
& 0.0125  & 1.5314e-03 & 1.939811357 & 3.2376e-05 & 5.774554268 \\ 
& 0.00625  & 3.8718e-04 & 1.983774825 & 1.0436e-08 & 11.59914019 \\ 
 \hline
  \multirow{4}{*}{400} & 0.1  & 7.8030e-02  & --- & 2.7348e-02 & --- \\  
& 0.05  & 3.0867e-02  & 1.337963627 & 1.3456e-02 & 1.023185719 \\
& 0.025  & 1.0337e-02 & 1.578247724 & 5.9488e-03 & 1.177579030\\ 
& 0.0125  & 2.9442e-03 & 1.811869965 & 8.9273e-04 & 2.736302862 \\ 
& 0.00625  & 7.6754e-04 & 1.939561834 & 1.6476e-05 & 5.759785971\\ 
 \hline
 \end{tabular}
}
}
\end{table}
\begin{figure}[H]
\centering
\begin{subfigure}[h]{0.32\textwidth}
\includegraphics[width=\textwidth]{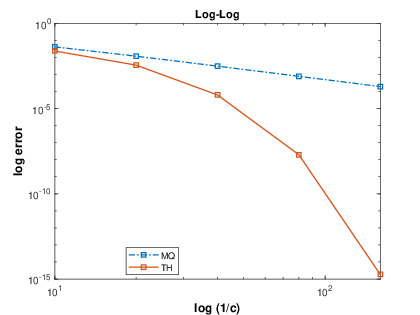}
\caption{}
\label{loglog100p}
\end{subfigure}
~ 
\begin{subfigure}[h]{0.32\textwidth}
\includegraphics[width=\textwidth]{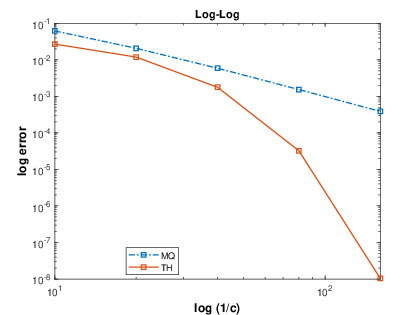}
\caption{}
\label{loglog200p}
\end{subfigure}
~ 
\begin{subfigure}[h]{0.32\textwidth}
\includegraphics[width=\textwidth]{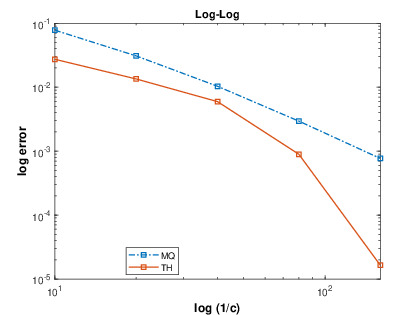}
\caption{}
\label{loglog400p}
\end{subfigure}
\caption{$\log|\mbox{error}|$ versus $\log\left(1/c\right)$ for $n=100$ (a), $n=200$ (b), and $n=400$ (c).}\label{loglog}
\end{figure}
\subsection{New transcendental RBF}
Let us introduce the following globally supported and infinitely differentiable transcendental RBF
\begin{eqnarray*}
\phi(r)=r\tanh\left({\frac{r}{c}}\right),
\end{eqnarray*}
abbreviated by RTH, where $r=\|x-x_j\|$ and $\|\cdot\|$ is the Euclidean norm $\mathbb{R}^d$.

The parameter $c>0$ is called {\it shape parameter} whose optimal value for getting accurate numerical solutions and good conditioning of the collocation matrix, can be found usually numerically.
\begin{theorem}
The  RTH RBF is conditionally negative definite of order $1$ on every $\mathbb{R}^d$.
\end{theorem}
\begin{proof}
We show that $\psi(r)=-\phi(r)$ is conditionally positive definite of order $1$.  We have $\psi(r)=f(s)=-\sqrt{s}\tanh\left(\frac{\sqrt{s}}{c}\right)$, where $s=r^2.$
Now for $$g(s)=-f'(s)=\frac{1}{2}s^{-\frac{1}{2}}\tanh\left(\frac{\sqrt{s}}{c}\right)+\frac{1}{2c}\left(1-\tanh^2\left(\frac{\sqrt{s}}{c}\right)\right),$$
we have
$$(-1)^{l}g^{(l)}(s)\geq 0,\quad \mbox{for}~\mbox{all}~l\in\mathbb{N}_0~\mbox{and}~\mbox{all}~s>0.$$
So $-f'(s)$ is completely monotone on $(0,\infty)$. Now, since $f\not\in\Pi_m^d$ , the claim is proved according to Micchelli's theorem \cite{5high}.
\end{proof}
\begin{remark}
Since $\phi$ is  conditionally negative definite of order $1$ and $\phi(0)=0$, then the matrix $A=\left[\phi(\|x_i-x_j\|)\right]_{1\leq i,j\leq n}$ has one positive and $n-1$ negative eigenvalues and in particular it is invertible.
\end{remark}
In the sequel,  we consider $d=1$, since our work is confined to the univariate case.
We have seen before that the RTH RBF is an smooth approximant to $\tau(r)=r$ with higher accuracy and better convergence properties than the MQ RBF $\sqrt{r^2+c^2},$ by decreasing shape parameter $c$. In Figure \ref{SC1}, we have plotted both RTH basis
\begin{eqnarray}\label{htphi}
\phi_j(x)=\left(x-x_j\right)\tanh\left(\frac{x-x_j}{c}\right),
\end{eqnarray}
and MQ basis (\ref{meq0}) centered  at $x_j=0$. It can be noted from Figure \ref{SC1} that the RTH RBF approaches to $|x|$ faster than the MQ RBF, even with larger shape parameters. Moreover,  in RTH RBF $\phi_j(x_j)=0$ independent of the value of $c$, but MQ requires that $c=0$. This property of the RTH RBF leads to getting more accurate results in corresponding quasi-interpolants. 
\begin{figure}[H]
\centering
\begin{subfigure}[h]{0.48\textwidth}
\includegraphics[width=\textwidth]{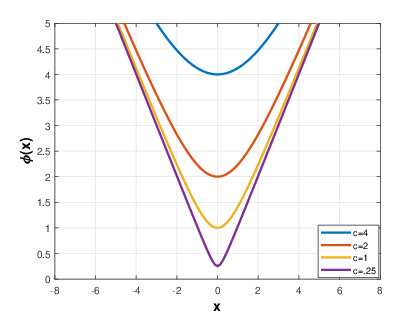}
\end{subfigure}
~ 
\begin{subfigure}[h]{0.48\textwidth}
\includegraphics[width=\textwidth]{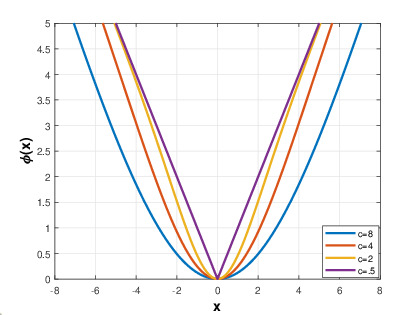}
\end{subfigure}
\caption{{\footnotesize{Plots of  Multiquadric RBF (left), and
RTH RBF (right) for different values of shape parameter
$c$.}}}\label{SC1}
\end{figure}
The first and second derivatives of the RTH RBF (\ref{htphi}) are of the form
\begin{align}
\phi^{'}_j(x)&=\tanh \left( {\frac {x-x_j}{c}} \right) +{\frac {(x-x_j)}{c} \left( 1- \left( \tanh \left( {\frac {x-x_j}{c}} \right)  \right) ^{2} \right) },\notag\\
\phi^{''}_j(x)&=2\,{\frac {1}{c} \left( 1- \left( \tanh \left( {\frac {x-x_j}{c}} \right) \right) ^{2} \right) }-2\,{\frac {(x-x_j)}{{c}^{2}}\tanh \left( {\frac {x-x_j}{ c}} \right)  \left( 1- \left( \tanh \left( {\frac {x-x_j}{c}} \right) \right) ^{2} \right)}\notag,
\end{align}
and are plotted in Figure \ref{fst1} for $c=1$.
\begin{figure}[H]
\centering
\begin{subfigure}[h]{0.31\textwidth}
\includegraphics[width=\textwidth]{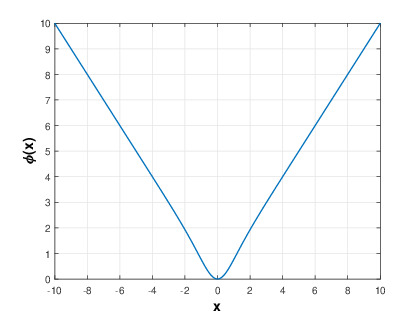}
\caption{}
\label{fig:Slope field}
\end{subfigure}
~ 
\begin{subfigure}[h]{0.31\textwidth}
\includegraphics[width=\textwidth]{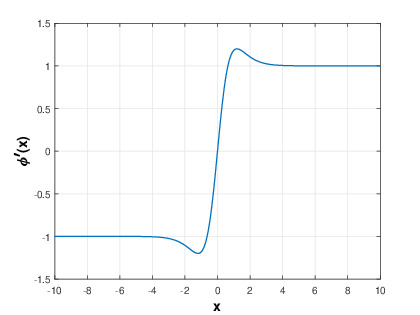}
 \caption{}
\end{subfigure}
~ 
\begin{subfigure}[h]{0.31\textwidth}
\includegraphics[width=\textwidth]{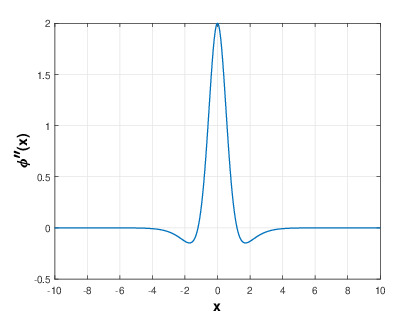}
\caption{}
\end{subfigure}
\caption{RTH RBF (a), first derivative (b), and second derivative of RTH RBF (c) with shape parameter $c=1$}\label{fst1}
\end{figure}
In Tables \ref{RBF} and \ref{RBF2}, we summarized the properties of both MQ and RTH RBFs, where $\xi=1.199678640,$ is obtained numerically by calculating the roots of the second derivative.
\begin{table}[H]
  \centering
   \caption{\footnotesize{Comparing both RBFs.}}\label{RBF}
  \scriptsize{
\begin{tabular}{|l|l|l|l|l|l|}
\hline
 & & & & &\\
\mbox{Name} & $\phi_{j}(x)$ & $ \displaystyle \lim_{x\rightarrow x_{j}}\phi_{j}(x) $ & $\displaystyle \lim_{c\rightarrow 0}\phi_{j}(x)$& $\displaystyle \lim_{x\rightarrow \pm \infty} \phi^{'}_{j}(x)$ & \mbox{condition} \\
 & & & & &\\
\hline
 & & & & &\\
\mbox{MQ RBF } & $\sqrt{c^{2}+(x-x_{j})^{2}}$ & $c$&$\lvert x-x_{j} \rvert$& $\pm 1$  & $x\in (-\infty,\infty)$\\
 & & & & &\\
 \hline
 & & & & &\\
\mbox{RTH RBF} & $\left(x-x_{j} \right)\tanh\left(\dfrac{ x-x_{j}}{c}\right)$ & $0$&$\lvert x-x_{j}\rvert$& $\pm 1$ & $x\in (-\infty,\infty)$ \\
 & & & & &\\
\hline
\end{tabular}
}
\end{table}
\begin{table}[H]
  \centering
   \caption{\footnotesize{Comparing both RBFs.}}\label{RBF2}
  \scriptsize{
\begin{tabular}{|l|l|l|l|l|}
\hline
 & & & &\\
\mbox{Name} & $\phi_{j}(x)$ & $\phi^{'}_{j}(x)$ & $\phi^{''}_{j}(x)$ & \mbox{condition} \\
 & & & &\\
\hline
 & & & &\\
\mbox{MQ RBF} & $\sqrt{c^{2}+(x-x_{j})^{2}}$ & Strictly increasing &$\geq 0$  & $x\in (-\infty,\infty)$\\
 & & & &\\
 \hline
 & & & &\\
\mbox{RTH RBF} & $\left(x-x_{j} \right)\tanh\left(\dfrac{ x-x_{j}}{c}\right)$& Strictly increasing &  $\geq 0$ &$x\in [-c \xi,c \xi] $ \\
 & &  & &\\
\hline
\end{tabular}
}
\end{table}
\subsection{Quasi-interpolation operator}
The quasi-interpolation operator of a function $f:[a, b]\rightarrow\mathbb{R}$ with RTH RBF on the scattered points 
\begin{eqnarray}\label{points}
a=x_{0}<x_{2}<\dotsb< x_{n} = b \qquad h := \max_{2\leq j\leq n} (x_{j} - x_{j-1}),
\end{eqnarray}
has the form 
\begin{equation}\label{meq04}
(\mathcal{L} _{RTH} f)(x)=f_{0}\alpha_{0}(x)+f_{1}\alpha_{1}(x)+\sum_{j=2}^{n-2}f_{j}\psi_{j}(x)+f_{n-1}\alpha_{n-1}(x)+f_{n}\alpha_{n}(x)
\end{equation}
where
\begin{align}
\alpha_{0}(x)&=\dfrac{1}{2}+\dfrac{\phi_{1}(x)-(x-x_{0})}{2(x_{1}-x_{0})}, \notag \\
\alpha_{1}(x)&=\dfrac{\phi_{2}(x)-\phi_{1}(x)}{2(x_{2}-x_{1})}-\dfrac{\phi_{1}(x)-(x-x_{0})}{2(x_{1}-x_{0})}, \notag\\
\alpha_{n-1}(x)&=\dfrac{(x_{n}-x)-\phi_{n-1}(x)}{2(x_{n}-x_{n-1)}}-\dfrac{\phi_{n-1}(x)-\phi_{n-2}(x)}{2(x_{n-1}-x_{n-2})}, \notag\\
\alpha_{n}(x)&=\dfrac{1}{2}+\dfrac{\phi_{n-1}(x)-(x_{n}-x)}{2(x_{n}-x_{n-1})}, \notag \\
\phi_j(x)&=\left(x-x_j\right)\tanh\left(\frac{x-x_j}{c}\right),\quad j=1,\ldots, n-1, ~c\in \mathbb{R}_+,\notag\\
\psi_{j}(x)&=\dfrac{\phi_{j+1}(x)-\phi_{j}(x)}{2(x_{j+1}-x_{j})}-\dfrac{\phi_{j}(x)-\phi_{j-1}(x)}{2(x_{j}-x_{j-1)}},\qquad 2\leq j \leq n-2.\notag
\end{align}
The formula (\ref{meq04}) can be rewritten as
\begin{eqnarray}\label{divform}
\left(\mathcal{L} _{RTH} f\right)(x)&=&\frac{1}{2}\sum_{j=1}^{n-1}f[x_{j-1},x_j,x_{j+1}](x_{j+1}-x_{j-1})\phi_j(x)+\\ \nonumber
& &\frac{f_0+f_n}{2}+\frac{1}{2}f[x_0,x_1](x-x_0)-\frac{1}{2}f[x_{n-1},x_n](x_n-x).
\end{eqnarray}
Let $\phi_{-1}(x)=|x-x_{-1}|,~\phi_{0}(x)=|x-x_{0}|,~\phi_{n}(x)=|x-x_{n}|$ and $\phi_{n+1}(x)=|x-x_{n+1}|$, then for $x\in [x_0,x_n]$, the operator 
$\mathcal{L} _{RTH}$ can be rearranged as
\begin{eqnarray}\label{meqpsi}
(\mathcal{L} _{RTH} f)(x)=\sum_{j=0}^nf_j\psi_j(x),
\end{eqnarray}
where 
\begin{eqnarray*}
\psi_j(x)=\frac{\phi_{j+1}(x)-\phi_j(x)}{2(x_{j+1}-x_j)}-\frac{\phi_{j}(x)-\phi_{j-1}(x)}{2(x_{j}-x_{j-1})},\quad j=0,\ldots,n,
\end{eqnarray*}
and $x_{-1}<x_0,~x_{n+1}>x_n.$
\begin{remark}\label{rem}
From relation (\ref{divform}), it is clear that the quasi-interpolation operator $\mathcal{L}_{RTH}$ reproduces the linear polynomials on $[x_0,x_n]$, that is
\begin{eqnarray}\label{linrep}
\sum_{j=0}^n(ax_j+b)\psi_j(x)=ax+b,\quad a,b\in\mathbb{R}\,,
\end{eqnarray}
from which we also get $\displaystyle \sum_{j=0}^n \psi_j(x)=1$ at any point $x \in [x_0,x_n]$.
\end{remark}
In order to prove the shape-preserving property of the quasi-interpolation operator (\ref{meq04}), we give some important definitions and theorems from differential geometry (cf. e.g. \cite{o1997elementary}).
\begin{definition}
A differentiable plane curve $\alpha:(a,b)\rightarrow\mathbb{R}^2$ is said to be regular if its derivative never vanishes. That is
\begin{eqnarray*}
\forall t \in (a,b), \quad \quad \quad \quad
\alpha^\prime(t)=\left( \frac{d \alpha_{1}}{dt},\frac{d \alpha_{2}}{dt}\right)\neq (0,0).
\end{eqnarray*}
\end{definition}
\begin{theorem}
Let $C$ be a regular plane curve given by  $\alpha(t)$. Then the curvature $\kappa$ of $C$ at  $t$ is given by
$$\kappa[\alpha](t)=\left \| \alpha^\prime(t) \times \alpha^{\prime \prime}(t) \right \|/\left \| \alpha^\prime(t) \right \|^{3}.$$
\end{theorem}
\begin{definition}\label{cur1}
Let $f\in C^{2}[a,b]$. The curvature of the plane curve $y=f(x)$ is given by
$$\kappa(x)=\dfrac{\lvert  f^{\prime \prime}(x)\rvert}{(1+(f^{\prime}(x))^{2})^{\frac{3}{2}}}.$$
\end{definition}
\begin{theorem}[Fundamental theorem of plane curves]
Let $\alpha, \gamma:(a,b)\rightarrow\mathbb{R}^2$ be regular plane curves such that $\kappa[\alpha](t)=\kappa[\gamma](t)$ for all $t\in(a,b)$. Then there is an orientation-preserving Euclidean motion $F:\mathbb{R}^2\rightarrow\mathbb{R}^2$ such that $\gamma=F~o~\alpha$.
\end{theorem}

\begin{corollary}\label{fund}
Two unit-speed plane curves which have the same curvature differ only by a Euclidean motion.
\end{corollary}
\begin{theorem}
The quasi-interpolation operator $\mathcal{L} _{RTH}$ constructed by data points $\{(x_j,f_j)\}$, is monotonicity and convexity-preserving for $c$ small enough.
\end{theorem}
\begin{proof}
According to the Corollary \ref{fund}, it suffices to show that
\begin{eqnarray*}
\displaystyle \lim_{c\rightarrow 0}\lvert \kappa_{\mathcal{L} _{MQ}}(x)-\kappa_{\mathcal{L} _{RTH}}(x) \rvert=0.
\end{eqnarray*}
Let $x\neq x_j$, otherwise both quasi-interpolants (\ref{meq2}) and (\ref{meq04}) do not have first and second derivatives as $c$ approaches $0$.
Now, according to definition \ref{cur1}, we have
\begin{eqnarray*}
\kappa_{{\mathcal{L} _{MQ}}}(x)=\dfrac{\lvert  ({\mathcal{L} _{MQ}}f)^{\prime \prime}(x)\rvert}{\left(1+\left(({\mathcal{L} _{MQ}}f)^{\prime}(x)\right)^{2}\right)^{\frac{3}{2}}}.
\end{eqnarray*}
Since for MQ RBF,
\begin{eqnarray*}
\phi_{j}^{\prime \prime}(x)=\frac{c^2}{\left(c^2+(x-x_j)^2\right)^{3/2}},
\end{eqnarray*}
then
\begin{eqnarray*}
\lim_{c\longrightarrow 0} \phi_{j}^{\prime \prime}(x)=0.
\end{eqnarray*}
Moreover
\begin{eqnarray*}
(\mathcal{L}_{ MQ}f)^{\prime \prime}(x)=\dfrac{1}{2}\sum_{j=1}^{n-1}\left[\dfrac{f_{j+1}-f_{j}}{x_{j+1}-x_{j}}-\dfrac{f_{j}-f_{j-1}}{x_{j}-x_{j-1}}\right]\phi_{j}^{\prime \prime}(x),
\end{eqnarray*}
then
\begin{eqnarray*}
\lim_{c\longrightarrow 0} \kappa_{\mathcal{L}_{{{MQ}}}}(x)=0,
\end{eqnarray*}
which leads to
\begin{eqnarray*}
\forall \epsilon >0 \quad \exists \delta_{1}>0;\quad \lvert \,c\, \rvert < \delta_1 \Rightarrow \lvert \kappa_{\mathcal{L}_{{{MQ}}}}(x) \rvert < \epsilon.
\end{eqnarray*}
Similarly, for RTH RBF, we have
\begin{eqnarray*}
\lim_{c\longrightarrow 0} \phi_{j}^{\prime \prime}(x)=0,
\end{eqnarray*}
then
\begin{eqnarray*}
\lim_{c\longrightarrow 0} \kappa_{\mathcal{L}_{{{RTH}}}}(x)=0,
\end{eqnarray*}
which leads to
\begin{eqnarray*}
\forall \epsilon >0 \quad \exists \delta_{2}>0;\quad \lvert \,c\, \rvert < \delta_2 \Rightarrow \lvert \kappa_{\mathcal{L}_{{{RTH}}}}(x) \rvert < \epsilon.
\end{eqnarray*}
The proof completes by considering $\delta=\min\{\delta_{1},\delta_{2}\}$.
\end{proof}
\section{Accuracy of the quasi-interpolation operator $\mathcal{L}_{RTH}$}
In this section, we give an approximation order for the quasi-interpolation operator $\mathcal{L}_{RTH}.$
\begin{theorem}
Assume $f^{\prime\prime}$ is Lipschitz continuous. The quasi-interpolation operator $\mathcal{L}_{RTH}f$, at the point set (\ref{points}) as $h\rightarrow 0$, converges as follows
\begin{eqnarray}
\|f-\mathcal{L}_{RTH}f\|_\infty\leq kh^2,
\end{eqnarray}
where $k$ is independent of $h$ and $c$.
\end{theorem}
\begin{proof}
Let $t(y)$ be the local Taylor approximation of $f$ at $y$, that is 
\begin{eqnarray*}
t(y)=f(x)+f'(x)(y-x)\,,\;\; x\in [a,b]
\end{eqnarray*}
According to Remark \ref{rem}, we get
\begin{eqnarray*}
\sum_{j=0}^n(x-x_j)\psi_j(x)=0,\quad \sum_{j=0}^n\psi_j(x)=1.
\end{eqnarray*}
Then we get
\begin{eqnarray*}
\sum_{j=0}^n t(x_j)\psi_j(x)&=&\sum_{j=0}^n\left[f(x)+f'(x)(x_j-x)\right]\psi_j(x)\\
&=&f(x)\sum_{j=0}^n\psi_j(x)+f'(x)\sum_{j=0}^n(x-x_j)\psi_j(x)\\
&=&f(x).
\end{eqnarray*}
Since $f^{\prime \prime}(x)$ is Lipschitz continuous, then for every $x_1,~x_2\in[a,b]$, $|f^{\prime \prime}(x_1)-f^{\prime \prime}(x_2)|\leq c_0|x_1-x_2|,$ where $\displaystyle 0<c_0=\mbox{ess sup}_{a\leq x\leq b}|f^{\prime \prime\prime}(x)|$.
Now according to (\ref{divform}), we have
\begin{eqnarray*}
|\mathcal{L}_{RTH}f(x)-f(x)|&=&\left|\sum_{j=0}^n\left(f(x_j)-t(x_j)\right)\psi_j(x)\right|\\
&\leq &\frac{1}{2}\left|\sum_{j=1}^{n-1}\left(f[x_{j-1},x_j,x_{j+1}]-t[x_{j-1},x_j,x_{j+1}]\right)(x_{j+1}-x_{j-1})\phi_j(x)\right|\\
&&+c_1(x-x_0)^2+c_2(x_n-x)^2\\
&\leq & \frac{1}{4}\sum_{j=1}^{n-1}|f^{\prime \prime}(\xi)-f^{\prime \prime}(\eta)||\phi_j(x)\, (x_{j+1}-x_{j-1})|,\quad (\xi,\eta\in (x_{j-1},x_{j+1}))\\
&&+c_1(x-x_0)^2+c_2(x_n-x)^2\\
&\leq &\frac{1}{2}c_0h\sum_{j=1}^{n-1}|x-x_j|(x_{j+1}-x_{j-1})+c_1(x-x_0)^2+c_2(x_n-x)^2\\
&\leq & \frac{1}{2}c_0h\sum_{\displaystyle\mathop{j=1}_{|x-x_j|\leq h}}^{n-1}|x-x_j|(x_{j+1}-x_{j-1})+ \frac{1}{2}c_0h\sum_{\displaystyle\mathop{j=1}_{|x-x_j|> h}}^{n-1}|x-x_j|(x_{j+1}-x_{j-1})\\
&&+c_1(x-x_0)^2+c_2(x_n-x)^2\\
&\leq &4c_0h^3+c_0h\left(\int_{|x-t|>h}|x-t|~dt+O(h)\right)+c_1(x-x_0)^2+c_2(x_n-x)^2\\
&\leq &k_1h^3+k_2h^2+k_3{(b-a)}^2\\
&\leq &kh^2
\end{eqnarray*}
\end{proof}
\section{Numerical results}
In this section, we compare the accuracy of the quasi-interpolation operator $\mathcal{L}_{RTH}$ with that of Wu and Schaback, $\mathcal{L}_{MQ}$ (defined in (\ref{meq2})) for the approximation of five functions. We take equidistant center points and choose different shape parameters $c$ and also different step sizes $h$. The maximum absolute error norm
is then computed for comparing approximation accuracy. 
The rate of convergence is also computed by
$$r_h=\frac{\ln \left(\frac{E_{h_i}}{E_{h_{i-1}}}\right)}{\ln \left(\frac{h_i}{h_{i-1}}\right)},$$
where $E_{h_i}$ indicates the error of the quasi-interpolant $\mathcal{L}_{RTH}f$ corresponding to the parameter $h_i$. In all tests, we chose $m=200$ equidistant evaluation points.
\subsection{Test problem 1}
In the first test problem, we apply the RTH quasi-interpolation to approximate the function (cf. \cite{fass})
\begin{eqnarray*}
f_{1}(x)=\frac{\sinh (x)}{ 1+\cosh (x)},\quad x\in [-3,3].
\end{eqnarray*}
The results are shown in Tables \ref{t1}-\ref{t3}. In Tables \ref{t1}, \ref{t2}, and \ref{t3}, we set $h=0.1, 0.01, 0.001$, respectively, and $c=2h, h, 0.5h, 0.2h, 0.1h$, then we compute the $\|\mathcal{L}_{RTH}f-f\|_{\infty}$ and $\|\mathcal{L}_{MQ}f-f\|_{\infty}$. In Table \ref{t4}, we set 
$c=0.01$, $h=0.2, 0.1, 0.05, 0.025, 0.0125$, to observe the convergence rate $r_{h}$ of $\mathcal{L}_{RTH}f$ with the variation of $h$.
\begin{table}[H]
\centering
\caption{\footnotesize{Comparison of approximation accuracy of RTH and MQ quasi-interpolation for the test problem 1.}}\label{t1}
  \scriptsize{
\begin{tabular}{llllll}
  \hline\\
$c$&$0.2$&$0.1$&$0.05$&$0.02$&$0.01$\\
 \hline
$h$&$0.1$&$0.1$&$0.1$&$0.1$&$0.1$\\
$\|\mathcal{L}_{MQ}f-f\|_{\infty}$&$9.3\times 10^{-3}$&$3.1\times 10^{-3}$&$1.1\times 10^{-3}$&$3.8\times 10^{-4}$&$2.8\times 10^{-4}$\\
$\|\mathcal{L}_{RTH}f-f\|_{\infty}$&$2.9\times 10^{-3}$&$6.2\times 10^{-4}$&$7.1\times 10^{-5}$&$2.3\times  10^{-4}$&$2.4\times 10^{-4}$
 \end{tabular}
}
\end{table}
\begin{table}[H]
\centering
\caption{\footnotesize{Comparison of approximation accuracy of RTH and MQ quasi-interpolation for the test problem 1.}}\label{t2}
  \scriptsize{
\begin{tabular}{llllll}
  \hline\\
$c$&$0.02$&$0.01$&$0.005$&$0.002$&$0.001$\\
 \hline
$h$&$0.01$&$0.01$&$0.01$&$0.01$&$0.01$\\
$\|\mathcal{L}_{MQ}f-f\|_{\infty}$&$1.8\times 10^{-4}$&$5.3\times 10^{-5}$&$1.6\times 10^{-5}$&$3.7\times 10^{-6}$
&$1.4\times 10^{-6}$\\
$\|\mathcal{L}_{RTH}f-f\|_{\infty}$&$3.0\times 10^{-5}$&$6.3\times 10^{-6}$&$7.2\times 10^{-7}$&$1.7\times 10^{-9}$&$7.9\times 10^{-14}$
 \end{tabular}
}
\end{table}
\begin{table}[H]
\centering
\caption{\footnotesize{Comparison of approximation accuracy of RTH and MQ quasi-interpolation for the test problem 1.}}\label{t3}
  \scriptsize{
\begin{tabular}{llllll}
  \hline\\
$c$&$0.002$&$0.001$&$0.0005$&$0.0002$&$0.0001$\\
 \hline
$h$&$0.001$&$0.001$&$0.001$&$0.001$&$0.001$\\
$\|\mathcal{L}_{MQ}f-f\|_{\infty}$&$2.7\times 10^{-6}$&$7.5\times 10^{-7}$&$2.1\times 10^{-7}$&$4.6\times 10^{-8}$&$1.6\times 10^{-8}$\\
$\|\mathcal{L}_{RTH}f-f\|_{\infty}$&$3.0\times 10^{-7}$&$6.3\times 10^{-8}$&$7.2\times 10^{-9}$&$1.7\times 10^{-11}$&$ 1.1\times 10^{-15}$
 \end{tabular}
}
\end{table}
\begin{table}[H]
\centering
\caption{\footnotesize{Convergence rates of $\mathcal{L}_{RTH}f$ by using $c=0.01$, $h=0.2, 0.1, 0.05, 0.025, 0.0125$ for the test problem 1.}}\label{t4}
  \scriptsize{
\begin{tabular}{lllllll}
  \hline\\
$c$&$0.01$&$0.01$&$0.01$&$0.01$&$0.01$\\
 \hline
$h$&$0.2$&$0.1$&$0.05$&$0.025$&$0.0125$\\
$\|\mathcal{L}_{RTH}f-f\|_{\infty}$&$9.5\times 10^{-4}$&$2.4\times 10^{-4}$&$5.4\times 10^{-5}$&$5.1\times 10^{-6}$&$1.0\times 10^{-6}$\\
$r_h$&-&$1.9855$&$2.1657$&$3.4056$&$2.3028$
 \end{tabular}
}
\end{table}
\subsection{Test problem 2}
In this experiment we apply the RTH quasi-interpolation to approximate the function (again considered in \cite{fass})
\begin{eqnarray}
f_{2}(x)=\sin\left(\frac{x}{2}\right)-2\cos(x)+4\sin(\pi x),\quad x\in [-4,4].
\end{eqnarray}
The comparison results are shown in Tables \ref{t5}-\ref{t7}. In Tables \ref{t5}, \ref{t6}, and \ref{t7}, we set $h=0.1, 0.01, 0.001$, respectively, and $c=2h, h, 0.5h, 0.2h, 0.1h$, then we compute the $\|\mathcal{L}_{RTH}f-f\|_{\infty}$ and $\|\mathcal{L}_{MQ}f-f\|_{\infty}$. In Table \ref{t8}, we set 
$c=0.01$, $h=0.2, 0.1, 0.05, 0.025, 0.0125$, to observe the convergence rate $r_{h}$ of $\mathcal{L}_{RTH}f$ with the variation of $h$.
\begin{table}[H]
\centering
\caption{\footnotesize{Comparison of approximation accuracy of RTH and MQ quasi-interpolation for the test problem 2.}}\label{t5}
  \scriptsize{
\begin{tabular}{llllll}
  \hline\\
$c$&$0.2$&$0.1$&$0.05$&$0.02$&$0.01$\\
 \hline
$h$&$0.1$&$0.1$&$0.1$&$0.1$&$0.1$\\
$\|\mathcal{L}_{MQ}f-f\|_{\infty}$&$1.2$&$4.5\times 10^{-1}$&$1.7\times 10^{-1}$&$7.1\times 10^{-2}$&$5.4\times 10^{-2}$\\
$\|\mathcal{L}_{RTH}f-f\|_{\infty}$&$4.5\times 10^{-1}$&$1.2\times 10^{-1}$&$1.4\times 10^{-2}$&$4.5\times 10^{-2}$&$4.9\times 10^{-2}$
 \end{tabular}
}
\end{table}
\begin{table}[H]
\centering
\caption{\footnotesize{Comparison of approximation accuracy of RTH and MQ quasi-interpolation for the test problem 2.}}\label{t6}
  \scriptsize{
\begin{tabular}{llllll}
  \hline\\
$c$&$0.02$&$0.01$&$0.005$&$0.002$&$0.001$\\
 \hline
$h$&$0.01$&$0.01$&$0.01$&$0.01$&$0.01$\\
$\|\mathcal{L}_{MQ}f-f\|_{\infty}$&$3.0\times 10^{-2}$&$9.2\times 10^{-3}$&$2.9\times 10^{-3}$&$7.1\times 10^{-4}$&$2.8\times 10^{-4}$\\
$\|\mathcal{L}_{RTH}f-f\|_{\infty}$&$6.4\times 10^{-3}$&$1.4\times 10^{-3}$&$1.5\times 10^{-4}$&$3.7\times 10^{-7}$&$1.7\times 10^{-11}$
 \end{tabular}
}
\end{table}
\begin{table}[H]
\centering
\caption{\footnotesize{Comparison of approximation accuracy of RTH and MQ quasi-interpolation for the test problem 2.}}\label{t7}
  \scriptsize{
\begin{tabular}{llllll}
  \hline\\
$c$&$0.002$&$0.001$&$0.0005$&$0.0002$&$0.0001$\\
 \hline
$h$&$0.001$&$0.001$&$0.001$&$0.001$&$0.001$\\
$\|\mathcal{L}_{MQ}f-f\|_{\infty}$&$4.9\times 10^{-4}$&$1.4\times 10^{-4}$&$4.1\times 10^{-5}$&$9.0\times 10^{-6}$&$3.3\times 10^{-6}$\\
$\|\mathcal{L}_{RTH}f-f\|_{\infty}$&$6.4\times 10^{-5}$&$1.4\times 10^{-5}$&$1.5\times 10^{-6}$&$3.7\times 10^{-9}$&$1.7\times 10^{-13}$
 \end{tabular}
}
\end{table}
\begin{table}[H]
\centering
\caption{\footnotesize{Convergence rates of $\mathcal{L}_{RTH}f$ by using $c=0.01$, $h=0.2, 0.1, 0.05, 0.025, 0.0125$ for the test problem 2.}}\label{t8}
  \scriptsize{
\begin{tabular}{llllll}
  \hline\\
$c$&$0.01$&$0.01$&$0.01$&$0.01$&$0.01$\\
 \hline
$h$&$0.2$&$0.1$&$0.05$&$0.025$&$0.0125$\\
$\|\mathcal{L}_{RTH}f-f\|_{\infty}$&$2.0\times 10^{-1}$&$4.9\times 10^{-2}$&$1.1\times 10^{-2}$&$1.1\times 10^{-3}$&$2.7\times 10^{-4}$\\
$r_h$&-&$2.0101$&$2.0899$&$3.4262$&$2.0034$
 \end{tabular}
}
\end{table}
\subsection{Test problem 3}
Consider the function (see again \cite{fass})
\begin{eqnarray}
f_{3}(x) = 10e^{-x^{2}}+x^{2},\quad x\in [-3,3],
\end{eqnarray}
for approximating by the RTH quasi-interpolation operator.
The comparison results are shown in Tables \ref{t9}-\ref{t11}. In Tables \ref{t9}, \ref{t10}, and \ref{t11}, we set $h=0.1, 0.01, 0.001$, respectively, and $c=2h, h, 0.5h, 0.2h, 0.1h$, then we compute the $\|\mathcal{L}_{RTH}f-f\|_{\infty}$ and $\|\mathcal{L}_{MQ}f-f\|_{\infty}$. In Table \ref{t12}, we set 
$c=0.01$, $h=0.2, 0.1, 0.05, 0.025, 0.0125$, to observe the convergence rate $r_{h}$ of $\mathcal{L}_{RTH}f$ on varying $h$.
\begin{table}[H]
\centering
\caption{\footnotesize{Comparison of approximation accuracy of RTH and MQ quasi-interpolation for the test problem 3.}}\label{t9}
  \scriptsize{
\begin{tabular}{llllll}
  \hline\\
$c$&$0.2$&$0.1$&$0.05$&$0.02$&$0.01$\\
 \hline
$h$&$0.1$&$0.1$&$0.1$&$0.1$&$0.1$\\
$\|\mathcal{L}_{MQ}f-f\|_{\infty}$&$ 4.9\times 10^{-1}$&$2.0\times 10^{-1}$&$7.4\times 10^{-2}$&$3.1\times 10^{-2}$&$2.4\times 10^{-2}$\\
$\|\mathcal{L}_{RTH}f-f\|_{\infty}$&$2.2\times 10^{-1}$&$5.5\times 10^{-2}$&$6.4\times 10^{-3}$&$2.0\times 10^{-2}$&$2.1\times 10^{-2}$
 \end{tabular}
}
\end{table}
\begin{table}[H]
\centering
\caption{\footnotesize{Comparison of approximation accuracy of RTH and MQ quasi-interpolation for the test problem 3.}}\label{t10}
  \scriptsize{
\begin{tabular}{llllll}
  \hline\\
$c$&$0.02$&$0.01$&$0.005$&$0.002$&$0.001$\\
 \hline
$h$&$0.01$&$0.01$&$0.01$&$0.01$&$0.01$\\
$\|\mathcal{L}_{MQ}f-f\|_{\infty}$&$1.3\times 10^{-2}$&$4.0\times 10^{-3}$&$1.3\times 10^{-3}$&$3.1\times 10^{-4}$&$1.2\times 10 ^{-4}$\\
$\|\mathcal{L}_{RTH}f-f\|_{\infty}$&$2.8\times 10^{-3}$&$5.9\times 10^{-4}$&$6.7\times 10^{-5}$&$1.6\times 10^{-7}$&$7.4\times 10^{-12}$
 \end{tabular}
}
\end{table}
\begin{table}[H]
\centering
\caption{\footnotesize{Comparison of approximation accuracy of RTH and MQ quasi-interpolation for the test problem 3.}}\label{t11}
  \scriptsize{
\begin{tabular}{llllll}
  \hline\\
$c$&$0.002$&$0.001$&$0.0005$&$0.0002$&$0.0001$\\
 \hline
$h$&$0.001$&$0.001$&$0.001$&$0.001$&$0.001$\\
$\|\mathcal{L}_{MQ}f-f\|_{\infty}$&$2.1\times 10^{-4}$&$6.0\times 10^{-5}$&$1.8\times 10^{-5}$&$3.9\times 10^{-6}$&$1.4\times 10^{-6}$\\
$\|\mathcal{L}_{RTH}f-f\|_{\infty}$&$2.8\times 10^{-5}$&$5.9\times 10^{-6}$&$6.7\times 10^{-7}$&$1.6\times 10^{-9}$&$7.5\times 10^{-14}$
 \end{tabular}
}
\end{table}
\begin{table}[H]
\centering
\caption{\footnotesize{Convergence rates of $\mathcal{L}_{RTH}f$ by using $c=0.01$, $h=0.2, 0.1, 0.05, 0.025, 0.0125$; Test Problem 3.}}\label{t12}
  \scriptsize{
\begin{tabular}{llllll}
  \hline\\
$c$&$0.01$&$0.01$&$0.01$&$0.01$&$0.01$\\
 \hline
$h$&$0.2$&$0.1$&$0.05$&$0.025$&$0.0125$\\
$\|\mathcal{L}_{RTH}f-f\|_{\infty}$&$8.6\times 10^{-2}$&$2.1\times 10^{-2}$&$5.0\times 10^{-3}$&$4.7\times 10^{-4}$&$2.6\times 10^{-4}$\\
$r_h$&-&$2.0085$&$2.0943$&$3.4210$&$2.0419$
 \end{tabular}
}
\end{table}
\begin{remark}
By analyzing the results in Tables \ref{t1}-\ref{t3}, \ref{t5}-\ref{t7}, and \ref{t9}-\ref{t11}, we see that the accuracy of the RTH quasi-interpolation scheme is
dependent on the shape parameter $c$ and on step size $h$. Furthermore, the accuracy of the RTH quasi-interpolation operator is
better than that of MQ for the same values of $c$ and $h$.
From Tables \ref{t4}, \ref{t8}, \ref{t12}, we see that the convergence rate of $\mathcal{L}_{RTH}$ reaches up to $2$ which justifies our theoretical findings of Section 3. 
By these numerical experiments, we can say that the quasi-interpolation $\mathcal{L}_{RTH}$ is a very attractive alternative, in terms of accuracy and convergence, to $\mathcal{L}_{MQ}$.
\end{remark}
\subsection{Test problem 4 (Runge function)}
Let us consider the Runge function on $[-1,1]$, that is $\displaystyle f_4(x)=\frac{1}{1+25x^{2}} $.
Figure \ref{runge-exapp} shows the exact and approximate values of $f_4$ for $c=0.01,~h=0.1,~0.02$. In Figure \ref{runge-exapp}, we see that the Runge phenomenon has disappeared by decreasing $h$. Relative errors are shown in Figure \ref{runge-err} using the RTH quasi-interpolation operator.
  \begin{figure}[H]
\centering
\begin{subfigure}[h]{0.40\textwidth}
\includegraphics[width=\textwidth]{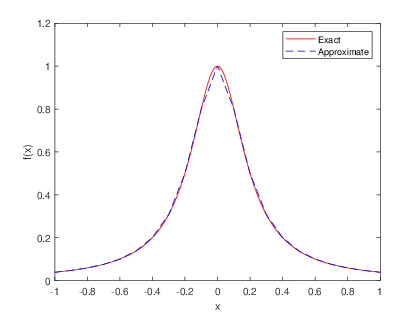}
\caption{}
\end{subfigure}
\begin{subfigure}[h]{0.40\textwidth}
\includegraphics[width=\textwidth]{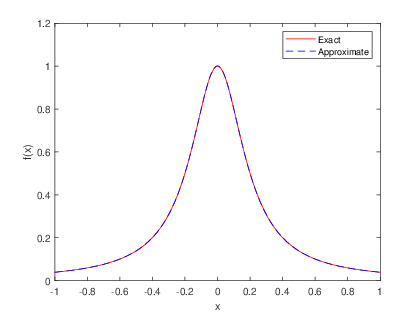}
 \caption{}
\end{subfigure}
\caption{RTH quasi-interpolation of  $f_4(x)=\frac{1}{1+25x^{2}};$ $h=0.1$ (a), $h=0.02$ (b), and $c=0.01.$}\label{runge-exapp}
\end{figure}  
 \begin{figure}[H]
\centering
\begin{subfigure}[h]{0.32\textwidth}
\includegraphics[width=\textwidth]{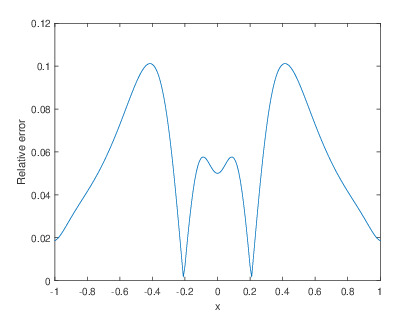}
\caption{}
\end{subfigure}
\begin{subfigure}[h]{0.32\textwidth}
\includegraphics[width=\textwidth]{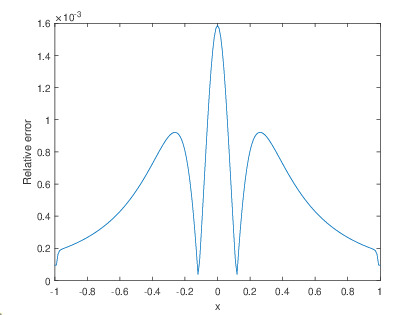}
 \caption{}
\end{subfigure}
\begin{subfigure}[h]{0.32\textwidth}
\includegraphics[width=\textwidth]{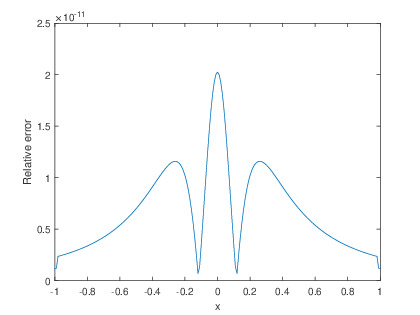}
\caption{}
\end{subfigure}
\caption{Relative errors: $c=0.1$ (a), $c=0.01$ (b), $c=0.001$ (c), and $h=0.02;$ for the test problem 4.}\label{runge-err}
\end{figure}
\subsection{Test problem 5 (Gibbs Phenomenon)}
It is well-known that any global or high order approximation method suffers from the Gibbs phenomenon if the function has a jump discontinuity in the given domain. In this test problem, we show that the RTH quasi-interpolation operator substantially mitigates the Gibbs phenomenon (cf. \cite{mohemerr}).
\begin{eqnarray*}
f_5(x)=\left\{\begin{array}{ll}
\frac{10}{3}x,& 0\leq x\leq 0.3,\\
1,&0.3\leq x\leq 0.6,\\
0,& 0.6<x\leq 1.
\end{array}
\right.
\end{eqnarray*}
Figure \ref{gib-exapp} shows the exact and approximate values of $f_5$. In Figure \ref{gib-exapp}, we see that the Gibbs
oscillations are considerably attenuated by decreasing $c$. Relative errors are reported in Figure \ref{gib-err}.
 \begin{figure}[H]
\centering
\begin{subfigure}[h]{0.32\textwidth}
\includegraphics[width=\textwidth]{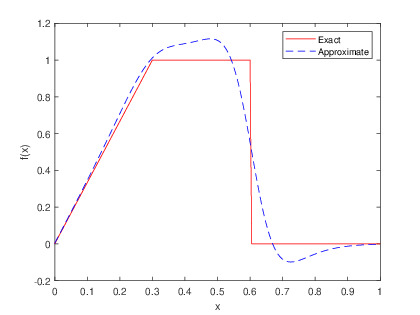}
\caption{}
\end{subfigure}
\begin{subfigure}[h]{0.32\textwidth}
\includegraphics[width=\textwidth]{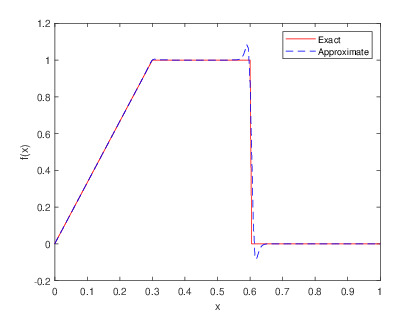}
 \caption{}
\end{subfigure}
\begin{subfigure}[h]{0.32\textwidth}
\includegraphics[width=\textwidth]{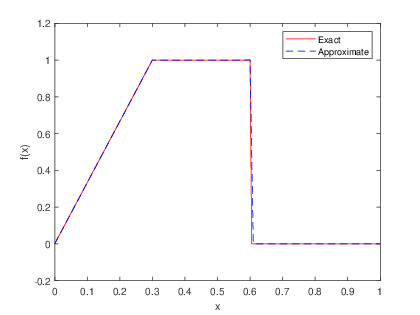}
\caption{}
\end{subfigure}
\caption{Approximations of $f_5$ with RTH quasi-interpolation; $c=0.1$ (a), $c=0.01$ (b),  $c=0.001$ (c), and $h=0.01.$}\label{gib-exapp}
\end{figure}
 \begin{figure}[H]
\centering
\begin{subfigure}[h]{0.32\textwidth}
\includegraphics[width=\textwidth]{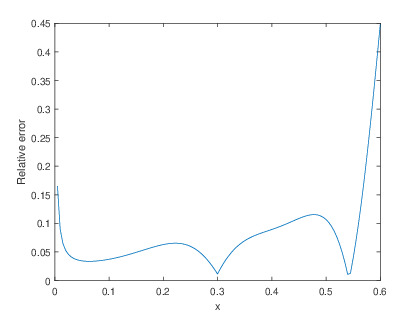}
\caption{}
\end{subfigure}
\begin{subfigure}[h]{0.32\textwidth}
\includegraphics[width=\textwidth]{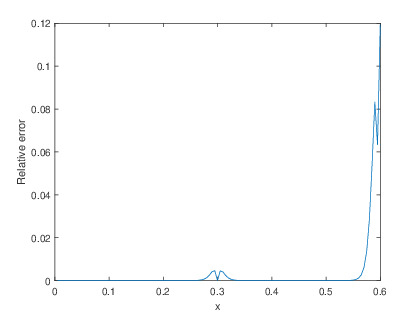}
 \caption{}
\end{subfigure}
\begin{subfigure}[h]{0.32\textwidth}
\includegraphics[width=\textwidth]{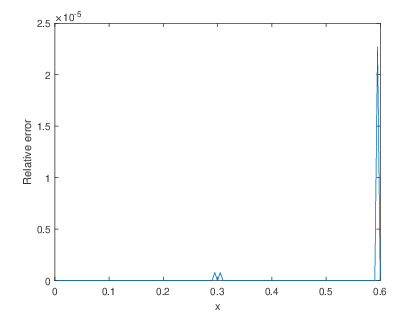}
\caption{}
\end{subfigure}
\caption{Relative errors: $c=0.1$ (a), $c=0.01$ (b),  $c=0.001$ (c), and $h=0.01;$ Test problem 5.}\label{gib-err}
\end{figure}
\subsection{Test problem 6 (A piecewise analytic function)}
As a final example, we consider the piecewise analytic function (cf. \cite{note})
\begin{eqnarray*}\label{af}
f_6(x)=\left\{\begin{array}{ll}
\sin(x),& x<0,\\
\cos(x),&x>0,
\end{array}
\right.
\end{eqnarray*} 
with $x\in[-1,1]$.
Figure \ref{gib-exapp2} shows the exact and approximate values of $f_6$, where Gibbs oscillations are considerably attenuated by decreasing $c$. Relative errors are shown in Figure \ref{gib-err2}.
 \begin{figure}[H]
\centering
\begin{subfigure}[h]{0.32\textwidth}
\includegraphics[width=\textwidth]{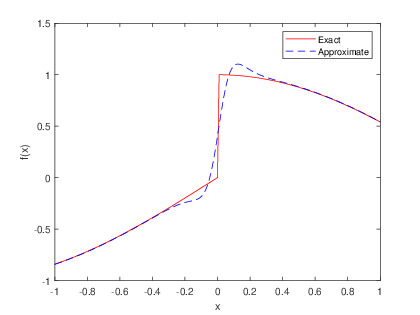}
\caption{}
\end{subfigure}
\begin{subfigure}[h]{0.32\textwidth}
\includegraphics[width=\textwidth]{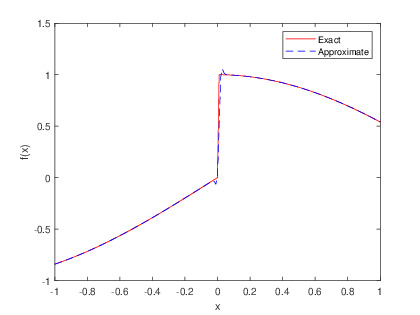}
 \caption{}
\end{subfigure}
\begin{subfigure}[h]{0.32\textwidth}
\includegraphics[width=\textwidth]{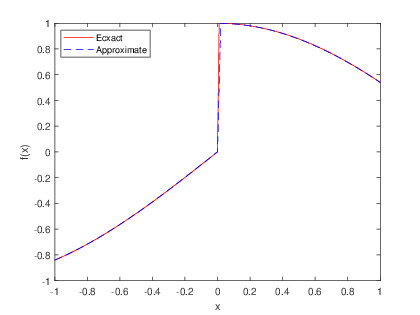}
\caption{}
\end{subfigure}
\caption{RTH quasi-interpolation of the piecewise analytic function $f_6$; $c=0.1$ (a), $c=0.01$ (b),  $c=0.001$ (c), and $h=0.02.$}\label{gib-exapp2}
\end{figure}
 \begin{figure}[H]
\centering
\begin{subfigure}[h]{0.32\textwidth}
\includegraphics[width=\textwidth]{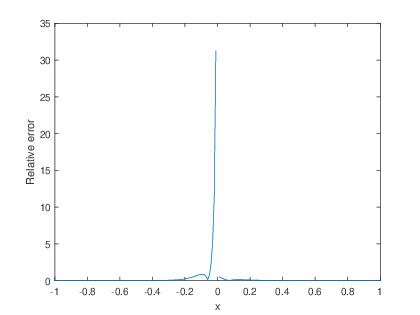}
\caption{}
\end{subfigure}
\begin{subfigure}[h]{0.32\textwidth}
\includegraphics[width=\textwidth]{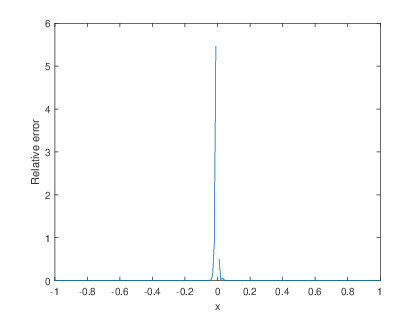}
 \caption{}
\end{subfigure}
\begin{subfigure}[h]{0.32\textwidth}
\includegraphics[width=\textwidth]{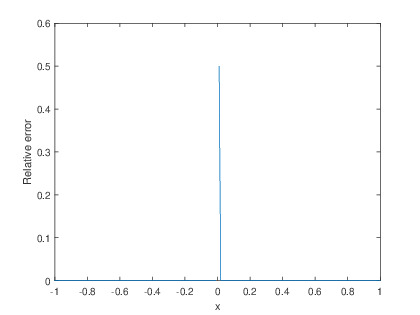}
\caption{}
\end{subfigure}
\caption{Relative errors: $c=0.1$ (a), $c=0.01$ (b),  $c=0.001$ (c), and $h=0.02;$ Test problem 6.}\label{gib-err2}
\end{figure}
\section{Conclusion}
In this paper, an efficient shape preserving quasi-interpolation operator with high degree of smoothness and very accurate results is proposed. It is based on the reformulation of Wu–Schaback’s quasi-interpolation operator by a new transcendental RBF of the form $\phi(r)=r\tanh\left({r \over c}\right)$. The 
quasi-interpolation operator, called ${\cal L}_{RTH}$ has nice convergence properties, being $\|{\cal L}_{RTH}-f\|_\infty \le k\, h^2$, with $h$ being the step size and $k$ a positive constant independent on the shape parameter $c$ and the step size $h$ (cf. Theorem 3.1).
Numerical experiments reveal that the proposed quasi-interpolation operator not only gives very accurate results but also it does not suffer of the Runge and Gibbs phenomena (see Test problems 4-6). \\ As a future work we are working in the application of the operator to real worlds problems, in particular to irregular surfaces approximation and image segmentation.
\vskip 0.2in
{\bf Acknowledgments}. The third author thanks for the support the GNCS-INdAM. His research has been done within the Italian Network on Approximation (RITA) and the thematic group on "Approximation Theory and Applications" of the Italian Mathematical Union.

\bibliographystyle{plain}
\bibliography{mybib}
\end{document}